\documentclass[11pt]{amsart}
\usepackage{amsmath, amssymb, array}
\usepackage{mathtools}
\usepackage[all]{xy}
\usepackage{ascmac}
\usepackage[dvips]{graphicx}
\usepackage{color}
\usepackage{amscd}
\usepackage{fancyhdr}
\usepackage[bookmarks=false]{hyperref} 
\usepackage{amsrefs}
\usepackage{cleveref}
\usepackage{chngcntr}
\usepackage{apptools}
\usepackage[titletoc,title]{appendix}
\newcommand{\ol}{\overline}

\newcommand{\rank}{\mathrm{rank}}
\newcommand{\Sel}{\mathrm{Sel}^{(2)}}

\newcommand{\Ker}{\mathrm{Ker}}
\newcommand{\im}{\mathrm{Im}}
\newcommand{\res}{\mathrm{res}}

\newcommand{\disc}{\mathrm{disc}}

\newcommand{\tors}{\mathrm{tors}}
\renewcommand{\1}{\mathbf{1}}
\newcommand{\fn}{\footnote}

 \DeclareFontFamily{U}{wncy}{}
    \DeclareFontShape{U}{wncy}{m}{n}{<->wncyr10}{}
    \DeclareSymbolFont{mcy}{U}{wncy}{m}{n}
    \DeclareMathSymbol{\Sha}{\mathord}{mcy}{"58} 
    
\title{Infinitely many hyperelliptic curves with exactly two rational points: Part II} 

\author{Hideki Matsumura}
\address[Hideki Matsumura]{Faculty of Science and Technology, Department of Mathematics, Keio University, 14-1, Hiyoshi 3-chome, Kouhoku-ku, Yokohama-shi, Kanagawa-ken, Japan}
\email{hidekimatsumura@keio.jp}

\thanks{This research was supported in part by KAKENHI 18H05233 as well as the KiPAS program FY2014--2018 of the Faculty of Science and Technology at Keio University. 
The author was supported by the Research Grant of Keio Leading-edge Laboratory of Science $\&$ Technology 2018--2019 (Grant Numbers 000053 and 000005).} 
\date{\today}

\makeatletter
  \def\section{\@startsection{section}{1}{\z@}%
     {4ex plus 0ex}%
     {4ex plus 0ex}%
     {\normalfont\large\bfseries}}
  \makeatother

\makeatletter
  \def\subsection{\@startsection{subsection}{1}{\z@}%
     {2ex plus 0ex}%
     {2ex plus 0ex}%
     {\normalfont\normalsize\bfseries}}
  \makeatother

\theoremstyle{plain}
 \newtheorem{theorem}{Theorem}[section] 
  \crefname{theorem}{Theorem}{Theorems}
 \newtheorem{proposition}[theorem]{Proposition}
 \crefname{proposition}{Proposition}{Propositions}
 \newtheorem{lemma}[theorem]{Lemma}
 \crefname{lemma}{Lemma}{Lemmas}
 \newtheorem{corollary}[theorem]{Corollary}
  \crefname{corollary}{Corollary}{Corollaries}
 \newtheorem{conjecture}[theorem]{Conjecture}
   \crefname{conjecture}{Conjecture}{Conjectures}
 
 \crefname{question}{Question}{Questions}
 
   \crefname{problem}{Problem}{Problems}
 \newtheorem{notation}[theorem]{Notation}
    \crefname{notation}{Notation}{Notations}
 
\theoremstyle{definition} 
 
  \crefname{definition}{Definition}{Definitions}
 
   \crefname{example}{Example}{Examples}
 \newtheorem{remark}[theorem]{Remark}
   \crefname{remark}{Remark}{Remarks}

\begin{document}


\maketitle


\begin{abstract}
In the previous paper, Hirakawa and the author determined the 
set of rational points of a certain infinite family 
of hyperelliptic curves $C^{(p;i,j)}$ parametrized by a prime number $p$ and integers $i$, $j$.
In the proof, we used the standard $2$-descent argument and a Lutz-Nagell theorem that was proven by Grant.
In this paper, we extend the above work. 
By using the descent theorem, the proof for $j=2$ is reduced to elliptic curves of rank $0$ that are independent of $p$.
On the other hand, 
 for odd $j$, we consider another hyperelliptic curve $C'^{(p;i,j)}$ whose Jacobian variety is isogenous to that of $C^{(p;i,j)}$,
and prove that the Mordell-Weil rank of the Jacobian variety of $C'^{(p;i,j)}$ is $0$ by $2$-descent. 
Then, we determine the set of rational points of $C^{(p;i,j)}$ by using the Lutz-Nagell type theorem.
\end{abstract}

\tableofcontents

\section{Main theorem}  
In arithmetic geometry, it is one of central problems to determine the sets of rational points on algebraic curves. 
In this paper, we study the sets of rational points on hyperelliptic curves.
A hyperelliptic curve defined over $\mathbb{Q}$ with a rational point can be embedded into its Jacobian variety. 
If $C$ is an elliptic curve, then $2$-descent 
makes it possible to bound the Mordell-Weil rank of 
$C$ by means of the $2$-Selmer group, and the Lutz-Nagell theorem makes it possible to determine the torsion points of 
$C$ by means of the discriminant. 
For example, by applying these arguments, we can prove that the only rational points on 
the elliptic curve defined by $y^2=x(x +p)(x-p)$
with a prime number $p \equiv 3 \pmod{8}$ are $(x,y) = (0,0)$, $(p,0)$, $(-p,0)$ and $\infty$, 
 i.e., such a prime number $p$ is never a congruent number (cf. \cite[D27]{Guy} and the references therein). 
It is also known that there exist infinitely many hyperelliptic curves with no or only one rational point for infinitely many genera \cite[Theorem 4]{NH}.
In the genus $2$ case, their proof is reduced to the fact that the Mordell-Weil rank of the Jacobian variety of a certain quotient of a Fermat curve $F(5)$ is $0$
(\cite[Proposition 4.1
]{GR}). 
In \cite[Corollary 1 and Theorem 2]{FR}, by using certain covering technique, Flynn and Redmond proved that for every prime number $p \equiv 5 \pmod{8}$, the hyperelliptic curve
defined by $y^2 = (x^2 +2p)(x^2 +3p)(x^2 +4p)$ has no nontrivial rational points. 

In this paper, we prove the following theorem. 

\begin{theorem} \label{MT} 
Let $p$ be a prime number, $i, j \in {\mathbb{Z}}$, and $C^{(p;i,j)}$ be the hyperelliptic curve defined by
\[y^2=x(x^2+2^ip^{j})(x^2+2^{i+1}p^{j}).\] 
\begin{enumerate}
\item 
For all $p$, we have $C^{(p;0,2)}(\mathbb{Q})=\{(0,0), \infty\}$.
\item If $p \equiv 3 \pmod{4}$, then
\[C^{(p;2,2)}(\mathbb{Q})=
\begin{cases}
\{(0,0), (6,\pm 216), \infty\} & (p = 3),\\
\{(0,0), \infty\} & (p \neq 3).
\end{cases}\]
\item If $p \equiv 13 \pmod{16}$, then $C^{(p;2,1)}(\mathbb{Q})=\{(0,0), \infty\}$.  
\item If $p \equiv 5 \pmod{16}$, then $C^{(p;2,3)}(\mathbb{Q})=\{(0,0), \infty\}$. 
\end{enumerate}
\end{theorem}

In the proof of \cref{MT} (1) and (2), we use the descent theorem (cf. \cite[Theorem 10, Example 9, 10]{Stoll2011}).
On the other hand, in the proof of \cref{MT} (3) and (4), we use a Richelot isogeny (cf. \cite{BM}, \cite{Nicholls}, \cite{Richelot}), the standard 2-descent argument (\cite{Schaefer98}) and a Lutz-Nagell type theorem that was proven by Grant (\cite[Theorem 3]{Grant}).

In the previous paper, Hirakawa and the author considered the following conditions.
\begin{enumerate}
\item $p \equiv 3 \pmod{16}$ and $(i,j)=(0,1)$.
\item $p \equiv 11 \pmod{16}$ and $(i,j)=(1,1)$.
\item $p \equiv  3 \pmod{8}$ and $(i,j)=(0,2)$.
\item $p \equiv - 3 \pmod{8}$ and $(i,j)=(0,2)$.
\end{enumerate}
Then, we proved that 
$C^{(p;i,j)}$ have exactly two obvious rational points (\cite[Theorem 1.1]{HM2}).

What about other cases?
In view of the abc conjecture, \cite{Granville}, \cite{Gusic} and \cite{PS}, 
we expect that $C^{(p;i,j)}$ has no nontrivial rational points (i.e. rational points except for $(0,0)$ and $\infty$) for all but finitely many $p$.
More precisely, we conjecture the following via computation by MAGMA \cite{Bosma-Cannon-Playoust}.

\begin{conjecture} \label{p=17}
Let $p$ be an odd prime number, $i, j \in {\mathbb{Z}}$, and $C^{(p;i,j)}$ be the hyperelliptic curve 
defined by the following equation.
 \[y^2=x(x^2+2^ip^{j})(x^2+2^{i+1}p^{j}).\]

\begin{enumerate}
\item 
Suppose that $(i,j)=(0,1)$. Then, 
\[C^{(p;i,j)}(\mathbb{Q})=
\begin{cases}
\{(0,0), (8,\pm 252), \infty\} & (p = 17),\\
\{(0,0), \infty\}  & (p \neq 17).
\end{cases}\]

\item Suppose that $(i,j)=(2,2)$. Then, 
\[C^{(p;i,j)}(\mathbb{Q})=
\begin{cases}
\{(0,0), (6,\pm 216), \infty\} & (p = 3),\\
\{(0,0), (5,\pm 375), \infty\}  & (p = 5),\\
\{(0,0), (136,\pm 235824), \infty\} & (p = 17),\\
\{(0,0), \infty\} & (p \neq 3, \; 5, \; 17).
\end{cases}\]
 
\item Suppose that $(i,j)=(2,3)$. Then, 
\[C^{(p;i,j)}(\mathbb{Q})=
\begin{cases}
\{(0,0), (72,\pm 45360), \infty\} & (p = 3),\\
\{(0,0), (98,\pm 115248), \infty\}  & (p = 7),\\
\{(0,0), \infty\} & (p \neq 3, \; 7).
\end{cases}\]

\item Suppose that $(i,j)=(0,2)$, $(1,1)$ or $(2,1)$. Then, 
\[C^{(p;i,j)}(\mathbb{Q})=\{(0,0), \infty\}.\]
\end{enumerate}
\end{conjecture} 

For each pair $(i,j)$, we have checked that there exist no additional rational points whose height of the $x$-coordinate is less than $10^5$ for $p<1000$.
\fn{We checked it by MAMGA. The commands are as follows.

(INPUTS)

$>$ P$<$x$>$:=PolynomialRing(Rationals()); 

$>$ for p in [3..1000] do;

$>$ if IsPrime(p) then;

$>$ C:=HyperellipticCurve(x*(x\^{}2+p)*(x\^{}2+2*p)); // (i,j)=(0,1).

$>$ Points(C: Bound:=10\^{}5);

$>$ end if;

$>$ end for;

(OUTPUTS)

\{@ (1 : 0 : 0), (0 : 0 : 1), (8 : -252 : 1), (8 : 252 : 1) @\} for $p=17$,

\{@ (1 : 0 : 0), (0 : 0 : 1) @\} for $p \neq 17$.
}

Note that there exist essentially six infinite families. 
\begin{remark} 
\begin{enumerate}
\item $C^{(p; i+4, j)}$ (resp.\ $C^{(p; i, j+4)}$) is isomorphic to $C^{(p; i, j)}$ via a map which maps $(x, y)$ to $(x/4, y/32)$ (resp.\ $(x/p^2, y/p^5))$. 
Moreover, we have isomorphisms of the following curves over $\mathbb{Q}$:
\fn{The isomorphisms are given by the following:
\begin{align*}
 C^{(p; 1, 3)} \overset{\simeq}{\to} C^{(p; 0, 1)};  \; & (x,y) \mapsto (2p^2/x,2p^2y/x^3).\\
 C^{(p; 1, 1)} \overset{\simeq}{\to} C^{(p; 0, 3)}; \;  & (x,y) \mapsto  (2p^2/x,2p^4y/x^3).\\
 C^{(p; 1, 2)} \overset{\simeq}{\to} C^{(p; 1, 6)}=C^{(p^2; 1, 3)} \overset{\simeq}{\to} C^{(p^2; 0, 1)}= C^{(p; 0, 2)}; \; & (x,y) \mapsto  (p^2x,p^5y) \mapsto (2p^2/x,2p^3y/x^3)\\
C^{(p; 2, 1)} \overset{\simeq}{\to} C^{(p; 3, 3)}; \;  & (x,y) \mapsto  (8p^2/x,32p^4y/x^3).\\
 C^{(p; 2, 3)} \overset{\simeq}{\to} C^{(p; 3, 1)};  \; & (x,y) \mapsto (8p^2/x,32p^2y/x^3).\\
  C^{(p; 2, 2)} \overset{\simeq}{\to} C^{(p; 2, 6)}=C^{(p^2; 2, 3)} \overset{\simeq}{\to} C^{(p^2; 3, 1)}= C^{(p; 3, 2)}; \; & (x,y) \mapsto  (p^2x,p^5y) \mapsto (8p^2/x,32p^3y/x^3).\\
  \end{align*}
} 
\begin{enumerate}
\item $C^{(p; 0, 1)} \simeq C^{(p; 1, 3)}$ 
\item $C^{(p; 0, 3)} \simeq C^{(p; 1, 1)}$ 
\item $C^{(p; 0, 2)} \simeq C^{(p; 1, 2)}$ 
\item $C^{(p; 2, 1)} \simeq C^{(p; 3, 3)}$  
\item $C^{(p; 2, 3)} \simeq C^{(p; 3, 1)}$  
\item $C^{(p; 2, 2)} \simeq C^{(p; 3, 2)}$ 
\end{enumerate} 
Our curves in \cite[Theorem 1.1]{HM2} are of the form $(a)$, $(b)$ and $(c)$, 
and our curves in \cref{MT} are of the form $(c)$ and $(f)$.
\item If $j=0$ or $p=2$, we obtain that $\rank (J^{(p;i,j)} (\mathbb{Q}))=0$ by MAGMA.
Therefore, $C^{(p;i,j)}(\mathbb{Q})=\{(0,0), \infty\}$ by \cite[Proposition 3.1]{HM2}. 
\end{enumerate}
\end{remark}

\cref{MT} proves \cref{p=17} in the following cases.

\begin{enumerate}
\item $(i,j)=(0,2)$.
\item $p \equiv 3 \pmod{4}$ and $(i,j)=(2,2)$.
\item $p \equiv  13 \pmod{16}$ and $(i,j)=(2,1)$.
\item $p \equiv 5 \pmod{16}$ and $(i,j)=(2,3)$.
\end{enumerate}

In the proof of \cite[Theorem 1.1]{HM2}, 
we first proved that the $\mathbb{F}_2$-dimension of the $2$-Selmer group $ \Sel(\mathbb{Q}, J^{(p; i, j)})$
is $2$ by $2$-descent \cite{Stoll2001}. 
Since the $\mathbb{F}_2$-dimension of the $2$-torsion part $J^{(p; i, j)}(\mathbb{Q})[2]$ is also $2$, 
we obtain that $\rank (J^{(p; i, j)}(\mathbb{Q}))=0$.
Next, we determined the set of rational points on $C^{(p; i, j)}$ which map to torsion points 
on $J^{(p; i, j)}$ via the Abel-Jacobi map by the Lutz-Nagell type theorem \cite[Theorem 3]{Grant}. 

On the other hand, computation by MAGMA 
suggests that except for cases $(i,j) = (0,2)$ and $p \equiv \pm 3 \pmod{8}$,
\[\dim_{\mathbb{F}_2} \Sel(\mathbb{Q}, J^{(p; i, j)})= 
\begin{cases}
3 & (p \equiv 3 \pmod{8}\; \mbox{and} \; (i,j)=(2,2)),\\
4 \; \mbox{or} \; 6 & (p \equiv 1 \pmod{8} \; \mbox{and} \; (i,j)=(0,2)),\\
4 & (\mbox{otherwise}). 
\end{cases} \]
 in \cref{MT}.
Thus, we can obtain only an upper bound
\[\rank (J^{(p; i, j)}(\mathbb{Q})) \leq 
\begin{cases}
1 & (p \equiv 3 \pmod{8}\; \mbox{and} \; (i,j)=(2,2)),\\
2 \; \mbox{or} \; 4 & (p \equiv 1 \pmod{8} \; \mbox{and} \; (i,j)=(0,2)),\\
2 & (\mbox{otherwise}). 
\end{cases} \]
in \cref{MT} by $2$-descent for $J^{(p; i, j)}$ itself. 
To prove \cref{MT} (1) and (2), 
we reduce the computations of the sets of rational points of $C^{(p;i,j)}$ 
 to elliptic curves of rank $0$ that are independent of $p$ by using the descent theorem \cite[Theorem 10, Example 9, 10]{Stoll2011}.
In particular, we can reprove \cite[Theorem 1.1 (3), (4)]{HM2} in easier way.

On the other hand, to prove \cref{MT} (3) and (4), we use a Richelot isogeny
combined with $2$-descent and the Lutz-Nagell type theorem.
A Richelot isogeny is a certain isogeny between the Jacobian varieties of hyperelliptic curves of genus $2$ 
(cf. \cite{BM}, \cite{Nicholls}, \cite{Richelot}).
The Jacobian variety $J'^{(p; i, j)}$ of the hyperelliptic curve $C'^{(p; i, j)}$ defined by
\[y^2=x(x^2-2^{i+2}p^j)(x^2-2^{i+3}p^j)\]
is isogenous to the Jacobian variety $J^{(p; i, j)}$ of $C^{(p; i, j)}$ by \cref{R-isog}.
Therefore, to prove that the Mordell-Weil rank of $J^{(p; i, j)}$ is $0$,
it is sufficient to prove that the Mordell-Weil rank of the Jacobian variety $J'^{(p; i, j)}$ of $C'^{(p; i, j)}$ is $0$.
We carry out this task by the standard $2$-descent argument \cite{Schaefer98}.
Then, we determine the set of rational points on $C^{(p; i, j)}$ which map to torsion points 
on $J^{(p; i, j)}$ via the Abel-Jacobi map by the Lutz-Nagell type theorem \cite[Theorem 3]{Grant}.  

In $\S 2$, we prove \cref{MT} (1) and (2) by determining the sets of rational points of certain elliptic curves.
In $\S 3$, we prove \cref{MT} (3) and (4). In $\S 3.1$, we outline the proof. 
Then, in $\S 3.2$ and $\S 3.3$, we prove that the Mordell-Weil rank of the Jacobian variety 
$J'^{(p; i, j)}$ of $C'^{(p; i, j)}$ is $0$ by $2$-descent \cite{Schaefer98}.


\section{Proof of \cref{MT} (1), (2)}
In this section, we reduce the computation of $C^{(p;i,j)}$ for $j=2$ to elliptic curves
by the descent theorem, and prove \cref{MT} (1) and (2).

\subsection{Case (1): $(i,j)=(0,2)$}

\begin{proof} [Proof of \cref{MT} $(1)$]
Consider
\[ C^{(p;0,2)} : y^2 =x(x^2 +p^2)(x^2 +2p^2).\]
Let $(x,y)=(U/W, V/W^3)$, $U$, $V$, $W \in \mathbb{Z}$. 
Then,
\[V^2=UW(U^2 + p^2W^2)(U^2 + 2p^2W^2).\]
We may assume that $(U,W)=1$.

\begin{enumerate}
\item Case I: $p \nmid U$. 

We claim that 
\[
d:=\gcd(UW,(U^2 + p^2W^2)(U^2 + 2p^2W^2))=2^l.\]
for some $l \in \mathbb{Z}_{\geq 0}$.
Indeed, if there is a prime number $q \neq 2$  such that $q \mid d$, then $q \mid U$ or $q \mid W$.
If $q  \mid U$, then $q \neq p$ by $p \nmid U$, and we have
\[0 \equiv (U^2+p^2W)(U^2+2p^2W) \equiv 2p^4W^2 \pmod{q}. \]
This implies that $q \mid W$, which contradicts $(U,W)=1$.
If $q  \mid W$, then 
\[0 \equiv (U^2+p^2W)(U^2+2p^2W) \equiv U^4 \pmod{q}. \]
This implies that $q \mid U$, which contradicts $(U,W)=1$.
If $p \mid \delta$, then
\[0 \equiv (U^2+p^2W)(U^2+2p^2W) \equiv U^4 \pmod{p}. \]
This implies that $p \mid U$, which is a contradiction.

Therefore, we have
\[\begin{cases}
\delta s^2=UW, \\
\delta t^2= (U^2+p^2W^2)(U^2+2p^2W^2) 
\end{cases}\]
with $\delta=1$ or $2$.

If $\delta=1$, then we have
\begin{align*} \label{EC1-1}
\left(\frac{t}{U^2} \right)^2=2 \left(\frac{pW}{U} \right)^4+3 \left(\frac{pW}{U} \right)^2+1. 
\end{align*}
By \cite[Proposition 1.2.1]{Connell}, 
if we let 
\begin{align*}
x &:= \frac{2 (\frac{t}{U^2}+1)}{(\frac{pW}{U})^2}, \\
y &:= \frac{4 (\frac{t}{U^2}+1)+6 (\frac{pW}{U})^2}{(\frac{pW}{U})^3}, \\
\end{align*}
then we obtain the following cubic model:
\[E_1: y^2=x^3+3x^2-8x-24. \]
By MAGMA, we have $\rank (E_1(\mathbb{Q}))=0$ and $E_1(\mathbb{Q})_{\tors} \simeq \mathbb{Z}/2\mathbb{Z}$.
Since $(-3,0)$, $\infty \in E_1(\mathbb{Q})$, we conclude that
\[E_1(\mathbb{Q})=\{(-3,0), \infty \}. \]
The rational points on $E_1$ correspond to $ \infty \in C^{(p;0,2)}(\mathbb{Q})$.

If $\delta=2$, then we have
\begin{align*} \label{EC1-2}
\left(\frac{t}{(pW)^2} \right)^2=\frac{1}{2} \left(\frac{U}{pW} \right)^4+\frac{3}{2} \left(\frac{U}{pW} \right)^2+1. 
\end{align*}
By \cite[Proposition 1.2.1]{Connell}, 
if we let 
\begin{align*}
x &:= \frac{2 (\frac{t}{(pW)^2}+1)}{(\frac{U}{pW})^2}, \\
y &:= \frac{4 (\frac{t}{(pW)^2}+1)+3 (\frac{U}{pW})^2}{(\frac{U}{pW})^3}, \\
\end{align*}
then we obtain the following cubic model:
\[E_2: y^2=x^3+\frac{3}{2}x^2-2x-3. \]
By MAGMA, we have $\rank (E_2(\mathbb{Q}))=0$ and $E_2(\mathbb{Q})_{\tors} \simeq \mathbb{Z}/2\mathbb{Z}$.
Thus, we conclude that
\[E_2(\mathbb{Q})=\{(-3/2,0), \infty \}. \]
The rational points on $E_2$ correspond to  $(0,0) \in C^{(p;0,2)}(\mathbb{Q})$.

\item Case II: $p \mid U$ (Thus, $p \nmid W$).

In this case, we have $p^3 \mid V$. Thus, $(U,V,W)=(pu,p^3v,w)$ with $(u,w)=1$, $p \nmid w$, and
\[pv^2=uw(u^2+w^2)(u^2+2w^2). \]

\begin{enumerate}

\item $p \mid uw$. 

As in case I, we know that
\[\gcd(uw,(u^2 + w^2)(u^2 + 2w^2)) =2^l.\]
for some $l \in \mathbb{Z}_{\geq 0}$.
Therefore we have
\[\begin{cases}
p \delta s^2=uw, \\
   \delta t^2= (u^2+w^2)(u^2+2w^2) 
\end{cases}\]
with $\delta=1$ or $2$.
These elliptic curves have rank $0$, and their rational points 
correspond to obvious ones $(0,0)$, $\infty \in C^{(p;0,2)}(\mathbb{Q})$ as in Case I.

\item $p \nmid uw$. 
\begin{enumerate}
\item Suppose that $p \equiv 7 \pmod{8}$.
Then, $u^2+w^2 \equiv 0 \pmod{p}$ or $u^2+2w^2 \equiv 0 \pmod{p}$,
which contradicts that $-1$ and $-2$ are quadratic non-residues modulo $p$. 

\item If $p \equiv 3 \pmod{8}$, then we can reprove \cite[Theorem 1.1 (3)]{HM2} in the above way.

Since $-2$ is a quadratic residue modulo $p$ and $-1$ is a quadratic non-residue modulo $p$, we have
\[\begin{cases}
p \delta s^2=u^2+2w^2, \\
   \delta t^2= uw(u^2+w^2)
\end{cases}\]
with $\delta=1$ or $2$. 

If $\delta=1$, then the second equation implies that
\[t^2= uw(u^2+w^2).\]
Since $w \neq 0$, we have
\[E_1: t'^2= u'(u'^2+1),\]
where we let $u'=u/w$ and $t'=t/w^2$.
By MAGMA, we have $\rank (E_1(\mathbb{Q}))=0$ and $E_1(\mathbb{Q})_{\tors} \simeq \mathbb{Z}/2\mathbb{Z}$.
Thus, we conclude that
\[E_1(\mathbb{Q})=\{(0,0), \infty \}. \]
$(0,0) \in E_1(\mathbb{Q})$ corresponds to 
 $(0,0) \in C^{(p;0,2)}(\mathbb{Q})$, and
$\infty \in E_1(\mathbb{Q})$ corresponds to $\infty \in C^{(p;0,2)}(\mathbb{Q})$.

If $\delta=2$, then the second equation implies that
\[2t^2= uw(u^2+w^2).\]
Since $w \neq 0$, we have
\[E_2: t'^2= u'(u'^2+1/4),\]
where we let $u'=u/2w$ and $t'=t/2w^2$.
By MAGMA, we have $\rank (E_2(\mathbb{Q}))=0$ and $E_2(\mathbb{Q})_{\tors} \simeq \mathbb{Z}/4\mathbb{Z}$.
Thus, we conclude that
\[E_2(\mathbb{Q})=\left\{\left(0,0 \right),\left(\frac{1}{2},\pm \frac{1}{2} \right), \infty \right\}. \]
$(0,0) \in E_2(\mathbb{Q})$ corresponds to 
$(0,0) \in C^{(p;0,2)}(\mathbb{Q})$ and 
$\infty \in E_2(\mathbb{Q})$ corresponds to $\infty \in C^{(p;0,2)}(\mathbb{Q})$.
$\left(\frac{1}{2},\pm \frac{1}{2} \right) \in E_2(\mathbb{Q})$ corresponds to $u=w$ i.e. $x=U/W=pu/u=p$. In this case, we have
\[y^2=p(x^2+p^2)(x^2+2p^2)=6p^5.\]
Since $6p \not\in \mathbb{Q}^2$, this is a contradiction.

\item
If $p \equiv -3 \pmod{8}$, then we can also reprove \cite[Theorem 1.1 (4)]{HM2} in the above way.

Since $-1$ is a quadratic residue modulo $p$ and $-2$ is a quadratic non-residue modulo $p$, we have
\[\begin{cases}
p \delta s^2=u^2+w^2, \\
   \delta t^2= uw(u^2+2w^2)
\end{cases}\]
with $\delta=1$ or $2$. 
In this case, we can prove that $\delta=1$.
Indeed, since $(u,w)=1$, either $u$ or $w$ is odd. 
If $u$ is even, then $w$ is odd and $u^2+w^2$ is odd. Thus, $\delta=1$.
If $u$ is odd and $w$ is even, then $u^2+w^2$ is odd. Thus, $\delta=1$.
If $u$ and $w$ are odd, then $uw(u^2+2w^2)$ is odd. Thus, $\delta=1$.

Therefore, we have
\[\begin{cases}
p s^2=u^2+w^2, \\
 t^2= uw(u^2+2w^2).
\end{cases}\]
By $u'=u/w$ and $t'=t/w^2$, we have the elliptic curve
As the above argument, we have
\[E_2: t'^2= u'(u'^2+2),\]
where we let $u'=u/2w$ and $t'=t/w^2$.
By MAGMA, we have $\rank (E(\mathbb{Q}))=0$ and $E(\mathbb{Q})_{\tors} \simeq \mathbb{Z}/2\mathbb{Z}$.
Thus, we conclude that
\[E(\mathbb{Q})=\{(0,0), \infty \}. \]
$(0,0) \in E(\mathbb{Q})$ corresponds to $(0,0) \in C^{(p;0,2)}(\mathbb{Q})$, and
$\infty \in E(\mathbb{Q})$ corresponds to $\infty \in C^{(p;0,2)}(\mathbb{Q})$.

\item If $p \equiv 1 \pmod{8}$, then we can also prove the same result in the above way.

Note that $-2$ and $-1$ are quadratic residue modulo $p$.
Since $p \nmid w$, we have $p \nmid (u^2+w^2)$ or $p \nmid (u^2+2w^2)$.

If $p \mid u^2+2w^2$ and $p \nmid u^2+w^2$, 
then we have
\[\begin{cases}
p \delta s^2=u^2+2w^2, \\
   \delta t^2= uw(u^2+w^2)
\end{cases}\]
with $\delta=1$ or $2$, and
we can prove the result in the same way as case (ii).

If $p \mid u^2+w^2$ and $p \nmid u^2+2w^2$, then we have
\[\begin{cases}
p  s^2=u^2+w^2, \\
    t^2= uw(u^2+2w^2),
\end{cases}\]
and we can prove the result in the same way as case (iii).
\end{enumerate}
\end{enumerate}
\end{enumerate}
\end{proof}

\subsection{Case (2): $p \equiv 3 \pmod {4}$ and $(i,j)=(2,2)$}

\begin{proof} [Proof of \cref{MT} $(2)$]
Consider
\[ C^{(p;2,2)} : y^2 =x(x^2 +4p^2)(x^2 +8p^2),\; 
 p \equiv 3 \pmod{4}. \]
Suppose that we have a rational point $(x,y) \in C^{(p;2,2)}(\mathbb{Q}) \setminus \{(0,0), \infty \}$.
Let $(x,y)=(U/W, V/W^3)$, $U$, $V$, $W \in \mathbb{Z}$. Then,
\[V^2=UW(U^2 + 4p^2W^2)(U^2 + 8p^2W^2).\]
We may assume that $(U,W)=1$.

\begin{enumerate}
\item Case I: $p \nmid U$.

As in the proof of \cref{MT} (1), we have
 \[\delta:=\gcd(UW,(U^2 + 4p^2W^2)(U^2 + 8p^2W^2)) =2^l\]
for some $l \in \mathbb{Z}_{\geq 0}$. Thus, we have
\[\begin{cases}
\delta s^2=UW, \\
\delta t^2= (U^2+4p^2W^2)(U^2+8p^2W^2) 
\end{cases}\]
with $\delta=1$ or $2$.
If $\delta=1$, then
\begin{align*} \label{EC2-1}
\left(\frac{t}{U^2} \right)^2=2 \left(\frac{2pW}{U} \right)^4+3 \left(\frac{2pW}{U} \right)^2+1. 
\end{align*}
By \cite[Proposition 1.2.1]{Connell}, 
if we let 
\begin{align*}
x &:= \frac{2 (\frac{t}{U^2}+1)}{(\frac{2pW}{U})^2}, \\
y &:= \frac{4 (\frac{t}{U^2}+1)+6 (\frac{2pW}{U})^2}{(\frac{2pW}{U})^3}, \\
\end{align*}
then we obtain the following cubic model:
\[E_1: y^2=x^3+3x^2-8x-24. \]
By MAGMA, we have $\rank (E_1(\mathbb{Q}))=0$ and $E_1(\mathbb{Q})_{\tors} \simeq \mathbb{Z}/2\mathbb{Z}$.
Since $(-3,0)$, $\infty \in E_1(\mathbb{Q})$, we conclude that
\[E_1(\mathbb{Q})=\{(-3,0), \infty \}. \]
The rational points on $E_1$ correspond to $\infty \in C^{(p;2,2)}(\mathbb{Q})$.

If $\delta=2$, then
\begin{align*} \label{EC2-2}
\left(\frac{t}{(2pW)^2} \right)^2=\frac{1}{2} \left(\frac{U}{2pW} \right)^4+\frac{3}{2} \left(\frac{U}{2pW} \right)^2+1. 
\end{align*}
By \cite[Proposition 1.2.1]{Connell}, 
if we let 
\begin{align*}
x &:= \frac{2 (\frac{t}{(2pW)^2}+1)}{(\frac{U}{2pW})^2}, \\
y &:= \frac{4 (\frac{t}{(2pW)^2}+1)+3 (\frac{U}{2pW})^2}{(\frac{U}{2pW})^3}, \\
\end{align*}
then we obtain the following cubic model:
\[E_2: y^2=x^3+\frac{3}{2}x^2-2x-3. \]
By MAGMA, we have $\rank (E_2(\mathbb{Q}))=0$ and $E_2(\mathbb{Q})_{\tors} \simeq \mathbb{Z}/2\mathbb{Z}$.
Since $(-3/2,0)$, $\infty \in E_2(\mathbb{Q})$, we conclude that
\[E_2(\mathbb{Q})=\{(-3/2,0), \infty \}. \]
The rational points on $E_2$ corresponds to $(0,0) \in C^{(p;2,2)}(\mathbb{Q})$.

\item Case II: $p \mid U$ (Thus, $p \nmid W$).

In this case, we have $p^3 \mid V$. Thus, $(U,V,W)=(pu,p^3v,w)$ with $(u,w)=1$, $p \nmid w$, and
\[pv^2=uw(u^2+4w^2)(u^2+8w^2). \]

\begin{enumerate}
\item $p \mid uw$.

As in case I, we know that
\[\delta:=\gcd(uw,(u^2 + 4w^2)(u^2 + 8w^2))=2^l.\]
for some $l \in \mathbb{Z}_{\geq 0}$.
Therefore we have
\[\begin{cases}
p \delta s^2=uw, \\
   \delta t^2= (u^2+4w^2)(u^2+8w^2) 
\end{cases}\]
with $\delta=1$ or $2$.
As above, we consider elliptic curves with rank $0$, and their rational points correspond to obvious ones $(0,0)$, $\infty \in C^{(p;2,2)}(\mathbb{Q})$ as in Case I.

\item $p \nmid uw$.
\fn{Note that we use the congruent condition of $p$ only here.} 

\begin{enumerate}
\item Suppose that $p \equiv 7 \pmod{8}$.

Then, $u^2+4w^2 \equiv 0 \pmod{p}$ or $u^2+8w^2 \equiv 0 \pmod{p}$,
which contradicts that $-1$ and $-2$ are quadratic non-residues modulo $p$. 

\item If $p \equiv 3 \pmod{8}$, then we can determine the sets of rational points of $C^{(p;2,2)}$ in the same way.

Since $-2$ is a quadratic residue modulo $p$ and $-1$ is a quadratic non-residue modulo $p$, we have
\[\begin{cases}
p \delta s^2=u^2+8w^2, \\
   \delta t^2= uw(u^2+4w^2)
\end{cases}\]
with $\delta=1$ or $2$. 

If $\delta=1$, then the second equation implies that
\[t^2= uw(u^2+4w^2).\]
Since $w \neq 0$, we have
\[E_1: t'^2= u'(u'^2+4),\]
where we let $u'=u/w$ and $t'=t/w^2$.
By MAGMA, we have $\rank (E_1(\mathbb{Q}))=0$ and $E_1(\mathbb{Q})_{\tors} \simeq \mathbb{Z}/4\mathbb{Z}$.
Thus, we conclude that
\[E_1(\mathbb{Q})=\{(0,0), (2, \pm 4), \infty \}. \]
$(0,0) \in E_1(\mathbb{Q})$ corresponds to $(0,0) \in C^{(p;2,2)}(\mathbb{Q})$, and
$\infty \in E_1(\mathbb{Q})$ corresponds to $\infty \in C^{(p;2,2)}(\mathbb{Q})$.
$(2,\pm 4) \in E_1(\mathbb{Q})$ corresponds to $u=2w$, i.e. $x=pu/w=2p$.
By $p s^2=u^2+8w^2=12w^2$, $s \in \mathbb{Q}$ if and only if $p=3$.
In this case, $x=6$ and this corresponds to $(6, \pm 216) \in C^{(p;2,2)}(\mathbb{Q})$.

If $\delta=2$, then the second equation implies that
\[2t^2= uw(u^2+4w^2).\]
Since $w \neq 0$, we have
\[E_2: t'^2= u'(u'^2+1),\]
where we let $u'=u/2w$ and $t'=t/2w^2$.
By MAGMA, we have $\rank (E_2(\mathbb{Q}))=0$ and $E_2(\mathbb{Q})_{\tors} \simeq \mathbb{Z}/2\mathbb{Z}$.
Thus, we conclude that
\[E_2(\mathbb{Q})=\{(0,0), \infty \}. \]
$(0,0) \in E_2(\mathbb{Q})$ corresponds to $(0,0) \in C^{(p;2,2)}(\mathbb{Q})$ and 
$\infty \in E_2(\mathbb{Q})$ corresponds to $\infty \in C^{(p;2,2)}(\mathbb{Q})$.
\end{enumerate}
\end{enumerate}
\end{enumerate}
\end{proof}


\section{$2$-descent}
In this section, we prove \cref{MT} (3) and (4) by using a Richelot isogeny, 
$2$-descent and the Lutz-Nagell type theorem.

Let $p$ be a prime number, $i, j \in {\mathbb{Z}}$, and $f(x)=x(x^2+2^ip^{j})(x^2+2^{i+1}p^{j})$.
Let $ C^{(p; i, j)}$ be the hyperelliptic curve defined by $y^2=f(x)$ and $J^{(p; i, j)}$ be its Jacobian variety.
First, we prove the following theorem by $2$-descent.
\begin{theorem} \label{rank=0}
Suppose that one of the following conditions holds.
\begin{enumerate}
\item $p \equiv 13 \pmod{16}$ and $(i,j)=(2,1)$.
\item $p \equiv 5 \pmod{16}$ and $(i,j)=(2,3)$.
\end{enumerate}
Then, we have $\rank(J^{(p; i, j)}(\mathbb{Q})) = 0$.
\end{theorem}
We treat the above two cases separately but in a similar manner in the following two subsections respectively.
We prove \cref{rank=0} by considering the Jacobian variety of another hyperelliptic curve $C'^{(p;i,j)}$ whose Jacobian variety is isogenous to that of $C^{(p;i,j)}$.  
\begin{theorem}  $($\cite{BM}, \cite[Theorem 5.7.4]{Nicholls}, \cite{Richelot}$)$ 
Let $C$ be the hyperelliptic curve over $\mathbb{Q}$ defined by $y^2=G_1(x)G_2(x)G_3(x)$, 
where $G_i(x)=g_{i2}x^2+g_{i1}x+g_{i0}$ and $\Delta:=\mathrm{det}(g_{ij})$.
Suppose that $\Delta \neq 0$, and let $C'$ be the hyperelliptic curve defined by the following equation.
\[ \Delta y^2=(G'_2G_3-G_2G'_3)(G'_3G_1-G_3G'_1)(G'_1G_2-G_1G'_2).\]
Here, $G'_i(x)$ denotes the derivative of $G_i(x)$.
Then, $J_C:=\mathrm{Jac}(C)$ and $J_{C'}:=\mathrm{Jac}(C')$ are isogenous over $\mathbb{Q}$.
In particular, 
\[\mathrm{rank}(J_C (\mathbb{Q}))=\mathrm{rank}(J_{C'}(\mathbb{Q})).\]
\end{theorem}
By this theorem, we obtain a defining equation of our $C'^{(p;i,j)}$.
\begin{corollary} \label{R-isog}
The Jacobian variety of $C^{(p;i,j)}$ is isogenous to that of $C'^{(p;i,j)}$ defined by the following equation.
\[y^2=x(x^2-2^{i+2}p^j)(x^2-2^{i+3}p^j).\]
\end{corollary}

If $i=2$ and $j$ is odd, the Jacobian variety of $C'^{(p;i,j)}$ is isomorphic to that of the hyperelliptic curve defined by 
\[y^2=x(x^2-16p^j)(x^2-32p^j) \]
by \cref{R-isog}.
The above curve is isomorphic to the hyperelliptic curve defined by
\[y^2=x(x^2-p^j)(x^2-2p^j),\]
which we also denote by $C'^{(p;i,j)}$.
Let $J'^{(p;i,j)}$ be the Jacobian variety of $C'^{(p;i,j)}$.
In what follows, we show that $\mathrm{rank}(J'^{(p;i,j)} (\mathbb{Q}))=0$.
Since $J'^{(p; i, j)}(\mathbb{Q)}/2J'^{(p; i, j)}(\mathbb{Q})$ can be embedded into the $2$-Selmer group $\Sel(\mathbb{Q},J'^{(p; i, j)})$, 
in order to bound the Mordell-Weil rank from above, it is sufficient to calculate the dimension of the $2$-Selmer group.
By \cite{Schaefer98}, we have
\begin{align*} 
\Sel(\mathbb{Q},J'^{(p; i, j)}) \simeq \{\alpha \in \Ker(N: L(S,2) \to \mathbb{Q}^{\times}/\mathbb{Q}^{\times 2}) \mid  \forall v \in S, \; \res_v(\alpha) \in \im(\delta_v) \}.
\end{align*} 
In each case, we can prove that the right hand side is generated by the image of $J(\mathbb{Q})[2]$. 
Here and after, we follow the notation in \cite{Schaefer98} and \cite{Stoll2001} as below.

\begin{notation} 
Suppose that $i=2$ and $j$ is odd.
We fix $p$, so we abbreviate $C'^{(p;i,j)}$ to $C$ and $J'^{(p;i,j)}$ to $J$. 
Let $y^2=f(x)$ be the defining equation of $C$.
Denote
\begin{itemize}
\item the $x$-coordinate of the point $P \in C(\mathbb{Q})$ by $x_{P}$,

\item every divisor class in $J(\mathbb{Q})$ represented by a divisor $D$ simply by $D$,

\item a fixed algebraic closure of $\mathbb{Q}$ by $\ol{\mathbb{Q}}$.

\end{itemize}
For every place $v$, we also use a similar notation and fix an embedding $\ol{\mathbb{Q}} \hookrightarrow \ol{\mathbb{Q}_v}$. 
Define
\begin{itemize} 
\item 
$L:=\mathbb{Q}[T]/(f(T)) \overset{\simeq}{\to}  
\prod_{k=1}^3 L^{(k)};$ $T \mapsto (T_1;T_2;T_3)$, where 
$L^{(1)}:=\mathbb{Q}[T_1]/(T_1)$, $L^{(2)}:=\mathbb{Q}[T_2]/(T_2^2-p^j)$ 
and $L^{(3)}:=\mathbb{Q}[T_3]/(T_3^2-2p^j)$.
We denote the trivial elements in $L^{\times}$ and $L^{\times}/{L^{\times 2}}$ by $\1$.

\item
$\delta:=x-T: J(\mathbb{Q}) \to L^{\times}/L^{\times 2};$ $D = \sum_{i = 1}^{n} m_{i} P_{i} \mapsto \prod_{i = 1}^{n}(x_{P_{i}}-T)^{m_{i}}$, where $D$ is a divisor whose support is disjoint from the support of the divisor $\mathrm{div}(y)$.
\end{itemize}
For every place $v$, define  
\begin{itemize} 
\item 
$L_v:=\mathbb{Q}_v[T]/(f(T)) \overset{\simeq}{\to} \prod_{k=1}^3 L_v^{(k)};$ 
$T \mapsto (T_1;T_2;T_3)$, where $L_v^{(1)}:=\mathbb{Q}_v[T_1]/(T_1)$, $L_v^{(2)}:=\mathbb{Q}_v[T_2]/(T_2^2-p^j)$ 
and $L_v^{(3)}:=\mathbb{Q}_v[T_3]/(T_3^2-2p^j)$.
For $v=\infty$, we fix isomorphisms
\begin{align*}
L_{\infty}^{(2)} &\simeq \mathbb{R} \times \mathbb{R}; \; T_2 \mapsto (\sqrt{p^j},-\sqrt{p^j}),\\
L_{\infty}^{(3)} &\simeq \mathbb{R} \times \mathbb{R}; \; T_3 \mapsto (\sqrt{2p^j},-\sqrt{2p^j}).
\end{align*}
We denote the trivial elements in $L_v^{\times}$ and $L_v^{\times}/{L_v^{\times 2}}$ by $\1_v$.

\item
$\delta_{v} :=(x-T)_v: J(\mathbb{Q}_{v}) \to L_v^{\times}/L_v^{\times 2};$ $D = \sum_{i = 1}^{n} m_{i} P_{i} \mapsto  \prod_{i = 1}^{n}(x_{P_{i}}-T)^{m_{i}}$, where $D$ is a divisor whose support is disjoint from the support of the divisor $\mathrm{div}(y)$.

\item
$\res_v: L^{\times}/L^{\times 2} \to L_v^{\times}/L_v^{\times 2}$ as the map induced by $L \to L_v;$ $T \mapsto T$.

\item $N:L^{\times}/L^{\times 2} \to \mathbb{Q}^{\times}/\mathbb{Q}^{\times 2}$ as the norm map.

\end{itemize}
Finally, define
\begin{itemize}

\item $S := \{2, p, \infty\}$.

\item $L(S,2):=\prod_{k=1}^3 L^{(k)}(S,2)$, where 
\[ L^{(k)}(S,2):=\{\alpha \in L^{(k) \times}/L^{(k) \times 2} \mid L^{(k)}(\sqrt{\alpha})/L^{(k)} \; \mbox{is unramified outside $S$} \}. \]
\item $\res_S:=\prod_{v \in S} \res_v: L^{\times}/L^{\times 2} \to \prod_{v \in S} L_v^{\times}/L_v^{\times 2}$.
\end{itemize}
\end{notation}


\subsection{Sketch of the proof of \cref{MT} (3), (4)}

In this subsection, we sketch the proof of \cref{MT} (3), (4).
By the straightforward calculation along the lines of \cite{Schaefer98}, we obtain the following tables for
$L^{(k)}(S,2)$ and $\im (\delta_v)$. 

\begin{lemma} \label{LS2}
In each case, the elements of the following table form a basis for $L^{(k)}(S,2)$.

\begin{table}[ht]
\begin{tabular}{cc} 
\begin{minipage}{0.57 \hsize}
\begin{tabular}{|c|c|c|} \hline
$k$ & basis of $L^{(k)}(S,2)$  & norm mod $\mathbb{Q}^{\times 2}$ \\ \hline \hline
$1$
 & $-1$ & $-1$ \\ \cline{2--6}
 & $2$ & $2$ \\ \cline{2--6} 
& $p$ & $p$ \\ \hline 
$2$
 & $-1$ & $1$ \\ \cline{2--6}
 & $\epsilon$ & $-1$ \\ \cline{2--6}
 & $2$ & $4$ \\ \cline{2--6} 
& $T_2$ & $-p$ \\ \hline  
$3$
 & $-1$ & $1$  \\ \cline{2--6}
 & $\epsilon'$ & $-1$ 
  \\ \cline{2--6}
 & $2$ & $4$ \\ \cline{2--6}  
& $T_3$ & $-2p$ \\ \hline 
\end{tabular}
\caption{$L^{(k)}(S,2)$ in Case (3)}
\end{minipage}
\begin{minipage}{0.57 \hsize}
%
\begin{tabular}{|c|c|c|} \hline 
$k$ & basis of $L^{(k)}(S,2)$  & norm mod $\mathbb{Q}^{\times 2}$ \\ \hline \hline
$1$
 & $-1$ & $-1$ \\ \cline{2--6}
 & $2$ & $2$ \\ \cline{2--6} 
& $p$ & $p$ \\ \hline 
$2$
 & $-1$ & $1$ \\ \cline{2--6}
 & $\epsilon$ & $-1$ \\ \cline{2--6}
 & $2$ & $4$ \\ \cline{2--6} 
& $T_2$ & $-p$ \\ \hline  
$3$
 & $-1$ & $1$  \\ \cline{2--6}
 & $\epsilon'$ & $-1$ 
 \\ \cline{2--6}
 & $2$ & $4$ \\ \cline{2--6} 
& $T_3/p$ & $-2p$ \\ \hline 
\end{tabular}
\caption{$L^{(k)}(S,2)$ in Case (4)}
\end{minipage}
\end{tabular}
\end{table}

Therefore, in case $(3)$, the following eight elements of $L^{\times}/L^{\times 2}$ form an $\mathbb{F}_2$-basis of $\Ker (N: L(S,2) \to \mathbb{Q}^{\times}/\mathbb{Q}^{\times 2})$:
\begin{align*}
(1;1;-1) && (1;1;2), &&  (1;-1;1), && (1;\epsilon;\epsilon'), \\ 
(1;2;1), && (-1;\epsilon;1), && (2;T_2;T_3) && (-p;T_2;1).
\end{align*}
In case (4), the following eight elements of $L^{\times}/L^{\times 2}$ form an $\mathbb{F}_2$-basis of $\Ker (N: L(S,2) \to \mathbb{Q}^{\times}/\mathbb{Q}^{\times 2})$:
\begin{align*}
(1;1;-1) && (1;1;2), &&  (1;-1;1), && (1;\epsilon;\epsilon'), \\ 
(1;2;1), && (-1;\epsilon;1), && (2;T_2;T_3/p), && (-p;T_2;1). 
\end{align*}
\end{lemma}

Next, we give the basis of $\im(\delta_v)$ and $J(\mathbb{Q}_v)/2J(\mathbb{Q}_v)$.
In each row for $J(\mathbb{Q}_v)/J(\mathbb{Q}_v)$, we give a polynomial in $x$ whose roots are the $x$-coordinates of points $R$ or $R_1+R_2$
such that $[R-\infty]$ or $[R_1+R_2-2\infty] \in J(\mathbb{Q}_v)$.
\begin{lemma} \label{Jv}
In each cases, the following elements form a basis of $J(\mathbb{Q}_v)/2J(\mathbb{Q}_v)$.
\begin{table}[ht]
\begin{tabular}{|c|rcl|} \hline 
$v$ & \multicolumn{1}{c}{$J(\mathbb{Q}_v)/2J(\mathbb{Q}_v)$} &   & \multicolumn{1}{c|}{$\im(\delta_v)$}  \\ \hline \hline
$p$ & $x$ & $\mapsto$ & $(2;T_2;T_3)$  \\ \cline{2--3} \cline{3--4} \cline{4--5}
& $x^2-p$ & $\mapsto$ & $(p; T_2;p)$ \\ \hline 
$\infty$ & $x$ & $\mapsto$ & $(1;(-1,1);(-1,1))$ \\ \cline{2--3} \cline{3--4} \cline{4--5} 
& $x+\sqrt{p}$ & $\mapsto$ & $(-1;(-1,-1);(-1,1))$ \\ \hline
$2$ & $x$ & $\mapsto$ & $(2;-T_2;-T_3)$ \\ \cline{2--3} \cline{3--4} \cline{4--5}
& $x^2-p$ & $\mapsto$ & $(3;T_2;-3)$ \\ \cline{2--3} \cline{3--4} \cline{4--5}
& $x-6$ & $\mapsto$ & $(6;6-T_2;6-T_3)$  \\ \cline{2--3} \cline{3--4} \cline{4--5}
& $x-5$ & $\mapsto$ & $(5;5-T_2;5-T_3)$ if $p \equiv 13 \pmod{32}$ \\ \cline{2--3} \cline{3--4} \cline{4--5}
& $x-13$ & $\mapsto$ & $(5;13-T_2;13-T_3)$ if $p \equiv 29 \pmod{32}$ \\ \hline
\end{tabular}
\caption{Generators of $J(\mathbb{Q}_v)/2J(\mathbb{Q}_v)$ and $\im(\delta_v)$ in Case (3)}
\begin{tabular}{|c|rcl|} \hline 
$v$ & \multicolumn{1}{c}{$J(\mathbb{Q}_v)/2J(\mathbb{Q}_v)$} &   & \multicolumn{1}{c|}{$\im(\delta_v)$}  \\ \hline \hline
$p$ & $x$ & $\mapsto$ &  $(2;T_2;T_3)$  \\ \cline{2--3} \cline{3--4} \cline{4--5}
& $x^2-p^3$ & $\mapsto$ & $(p; T_2;p)$ \\ \hline 
$\infty$ & $x$ & $\mapsto$ & $(1;(-1,1);(-1,1))$ \\ \cline{2--3} \cline{3--4} \cline{4--5} 
& $x+p\sqrt{p}$ & $\mapsto$ & $(-1;(-1,-1);(-1,1))$ \\ \hline
$2$ & $x$ & $\mapsto$ & $(2;-T_2;-T_3)$ \\ \cline{2--3} \cline{3--4} \cline{4--5}
& $x^2-p^3$ & $\mapsto$ & $(3;T_2;-3)$ \\ \cline{2--3} \cline{3--4} \cline{4--5}
& $x-6$ & $\mapsto$ & $(6;6-T_2;6-T_3)$  \\ \cline{2--3} \cline{3--4} \cline{4--5}
& $x-13$ & $\mapsto$ & $(5;13-T_2;13-T_3)$ if $p \equiv 5 \pmod{32}$ \\ \cline{2--3} \cline{3--4} \cline{4--5}
& $x-5$ & $\mapsto$ & $(5;5-T_2;5-T_3)$ if $p \equiv 21 \pmod{32}$ \\ \hline
\end{tabular}
\caption{Generators of $J(\mathbb{Q}_v)/2J(\mathbb{Q}_v)$ and $\im(\delta_v)$ in Case (4)}
\end{table}
\end{lemma}

From \cref{LS2,Jv} we can show that $\Sel(\mathbb{Q},J)$ is generated by $\delta(J(\mathbb{Q})[2])$.
\begin{lemma} \label{Sel2} 
In case (3), the following two elements of $L^{ \times}/L^{ \times 2}$ form an $\mathbb{F}_{2}$-basis of $\Sel(\mathbb{Q},J)$:
\begin{align*}  
t_1 &:= \delta((0,0)-\infty)=-T+(T^2-p)(T^2-2p)=(2;-T_2;-T_3),\\ 
t_2 &:= \delta \left(\sum_{\substack{P \in C(\ol{\mathbb{Q}}) \\ x_{P}^{2}-p = 0}}P - 2\infty \right) =(T^2-p)-T(T^2-2p)=(-p;T_2;p).
\end{align*}  
In case (4), the following two elements of $L^{ \times}/L^{ \times 2}$ form an $\mathbb{F}_{2}$-basis of $\Sel(\mathbb{Q},J)$:
\begin{align*}  
t_1 &:= \delta((0,0)-\infty)=-T+(T^2-p^3)(T^2-2p^3)=(2;-T_2;-T_3),\\ 
t_2 &:= \delta \left(\sum_{\substack{P \in C(\ol{\mathbb{Q}}) \\ x_{P}^{2}-p^3 = 0}}P - 2\infty \right) =(T^2-p^3)-T(T^2-2p^3)=(-p;T_2;p).
\end{align*}  
\end{lemma}
Therefore, by \cref{LS2,Jv,Sel2}, we obtain $\rank (J(\mathbb{Q}))=0$.
This completes the proof of \cref{rank=0}.

Finally, by taking \cref{rank=0} into account, \cref{MT} $(3)$, $(4)$ is an immediate consequence of the following proposition that we proved in \cite{HM2}.

\begin{proposition} [{\cite{HM2} Proposition 3.1}] \label{AJ}
Let $p$ be a prime number, $i, j \in {\mathbb{Z}}_{\geq 0}$, and $C^{(p;i,j)}$ be the hyperelliptic curve defined by
$y^2=x(x^2+2^ip^{j})(x^2+2^{i+1}p^{j})$ and $J$ be its Jacobian variety.
Let $\phi:C^{(p;i,j)} \to J$ be the Abel-Jacobi map defined by $\phi(P)=[P- \infty]$.
Let $P \in C^{(p;i,j)} (\mathbb{Q}) \setminus \{\infty\}$ be a rational point such that $\phi(P) \in J^{(p;i,j)}(\mathbb{Q})_{\mathrm{tors}}$. 
\begin{enumerate}
\item Suppose that $p \neq 3$. Then, $P=(0,0)$.
\item Suppose that $p =3$ and $(i,j) \not\equiv (2,2)$, $(3,2) \pmod{4}$.
Then, $P=(0,0)$.
\end{enumerate}
\end{proposition}

This proposition follows from the following Lutz-Nagell type theorem.
\begin{theorem} [{\cite[Theorem 3]{Grant}}] \label{GLN} 
Let $C$ be the hyperelliptic curve of odd degree defined by $y^2=f(x)$, and $J$ be its Jacobian variety.
Let $\phi:C \to J$ be the Abel-Jacobi map defined by $\phi(P)=[P- \infty]$.
Let $P \in C(\mathbb{Q}) \setminus \{\infty\}$ be a rational point such that $\phi(P) \in J(\mathbb{Q})_{\mathrm{tors}}$. 
 Then,
\begin{enumerate}
\item $a$, $b \in \mathbb{Z}$.
\item Either $b=0$ or $b^2 \mid \disc(f)$.
\end{enumerate}
\end{theorem} 

In the next two subsections, we prove \cref{Jv,Sel2}.


\subsection{Case (3): $p \equiv 13 \pmod{16}$ and $(i,j)=(2,1)$} 

Suppose that $(i,j)=(2,1)$, and $p \equiv 13 \pmod{16}$. 

We can compute the norm of the fundamental unit in Table 1, 2 by the following lemma.
\begin{lemma} \label{Norm} 
Let $K=\mathbb{Q}(\sqrt{2p})$ or $K=\mathbb{Q}(\sqrt{p})$.
Let $p \equiv 5 \pmod{8}$ be a prime number, and $\epsilon \in \mathcal{O}_K^{\times}$ be the fundamental unit.
Then, $N(\epsilon)=-1$.
\end{lemma}
\begin{proof}
We prove it for $K=\mathbb{Q}(\sqrt{2p})$ by contradiction. 
Suppose that $N(\epsilon)=1$. Then,
there exist $x$, $y \in \mathbb{Z}$  such that $x^2-2py^2=1$. 
We take $y$ to be minimal among such pairs of $x$ and $y$.
Then, $(x+1)(x-1)=2py^2$.
Since the left hand side is divisible by $2$, $\gcd(x+1,x-1)=2$.
We have four cases:
\begin{enumerate}
\item There exist $y_1$, $y_2$ such that 
\[\begin{cases}
y=2y_1y_2,\\
\gcd(y_1,y_2)=1,\\
x+1=2 \cdot 2py_1^2,\\
x-1=2  y_2^2.
\end{cases}\]
Then, $-1=y_2^2-2py_1^2=(y_2+y_1 \sqrt{2p})(y_2-y_1 \sqrt{2p})$.
Thus, $y_2+y_1 \sqrt{2p} \in \mathcal{O}_K^{\times}$ has norm $-1$.
This contradicts that the fundamental unit $\epsilon$ has norm $1$.

\item There exist $y_1$, $y_2$ such that 
\[\begin{cases}
y=2y_1y_2,\\
\gcd(y_1,y_2)=1,\\
x+1=2 \cdot 2y_1^2,\\
x-1=2  py_2^2.
\end{cases}\]
Then, $1=2y_1^2-py_2^2$.
This is a contradiction since $(2/p)=-1$.

\item There exist $y_1$, $y_2$ such that 
\[\begin{cases}
y=2y_1y_2,\\
\gcd(y_1,y_2)=1,\\
x+1=2 py_1^2,\\
x-1=2  \cdot 2y_2^2.
\end{cases}\]
Then, $-1=2y_2^2-py_1^2$.
This is a contradiction since $(-2/p)=-1$.

\item There exist $y_1$, $y_2$ such that 
\[\begin{cases}
y=2y_1y_2,\\
\gcd(y_1,y_2)=1,\\
x+1=2 y_1^2,\\
x-1=2  \cdot 2p y_2^2.
\end{cases}\]
Since $1=y_1^2-2py_2^2$, we have $y_2=y$ by the minimality of $y$.
Then, $y_1=1/2$ by the first equation, which contradicts $y_1 \in \mathbb{Z}$.
\end{enumerate}
The proof for $K=\mathbb{Q}(\sqrt{p})$ is similar.
\end{proof}

To show \cref{Jv}, we first calculate the $2$-torsion subgroups $J(\mathbb{Q}_v)[2]$.

\begin{lemma}  \label{2 torsion21}
The following two elements form an $\mathbb{F}_{2}$-basis of $J(\mathbb{Q})[2]$, $J(\mathbb{Q}_{2})[2]$ and $J(\mathbb{Q}_{p})[2]$: 
\begin{align*}
 (0,0) &- \infty, & \sum_{\substack{P \in C(\ol{\mathbb{Q}}) \\ x_{P}^{2}-p = 0}}P &- 2\infty.  
 \end{align*}
On the other hand, the following three elements form an $\mathbb{F}_{2}$-basis of $J(\mathbb{Q}_{\infty})[2]$:
\begin{align*}
 (0,0) &- \infty, & (\sqrt{p},0) &-\infty, & (-\sqrt{p},0) &-\infty,  & (\sqrt{2p}, 0)&-\infty.
 \end{align*}
In particular, we have the following table.
\begin{table}[ht]
    \begin{tabular}{|l||l|l|} \hline
  $v$   & $\dim J(\mathbb{Q}_{v})[2]$ & $\dim \im(\delta_{v})$ 
  \\ \hline \hline
     $2$   & $2$ & $4$      \\ \hline
     $p$   & $2$   & $2$       \\ \hline
      $\infty$   & $4$ & $2$   \\ \hline  
    \end{tabular}
\end{table}
\end{lemma}

\begin{proof}

The first and second statements follow from \cite[Lemma 5.2]{Stoll2014}.
Note that
\begin{itemize}
\item $p$ and $2p$ are not square in $\mathbb{Q}_{v}$ for $v=2$, $p$.
\item $2$ and $2p$ are square in $\mathbb{Q}_{\infty}$.
\end{itemize}
The third statement follows from the following formula (cf. \cite[p. 451, proof of Lemma 3]{FPS}).
\[\dim_{\mathbb{F}_2} \im(\delta_{v}) =\dim_{\mathbb{F}_2} J({\mathbb{Q}_{v}})[2]
\begin{cases}
+0 & (v \neq 2, \infty), \\
+2 & (v = 2), \\
-2 & (v = \infty). \\
\end{cases}\]
\end{proof}

We also use the following formula.
\begin{lemma} [{\cite[Lemma 2.2]{Schaefer}}]  \label{image of 2-torsion}
Let $C$ be the hyperelliptic curve over $\mathbb{Q}$ defined by $y^2=f(x)$ such that $\deg f$ is odd. 
For every place $v$ of $\mathbb{Q}$, any point on $J(\mathbb{Q}_v)$ can be represented by a divisor of degree $0$ whose support is disjoint from the support of the divisor $\mathrm{div}(y)$.
 Then, we have $\delta_{v}(D) = 1$ if $D$ is supported at $\infty$.
 If $D$ is of the form $D=\sum_{i=1}^{n} D_i$ with $D_i=(\alpha_i,0)$, where 
 $\alpha_i$ runs through all roots of a monic irreducible factor $h(x) \in \mathbb{Q}_v[x]$ of $f(x)$, then we have
 \[\delta_{v}(D)=(-1)^{\deg h} \left(h(T)-\frac{f(T)}{h(T)} \right).\] 
\end{lemma} 

.
\begin{proof} [{Proof of \cref{Jv} for Case (3)}]
$v=p$:
\cref{image of 2-torsion} implies that 
\begin{align*}
 \delta_p((0,0)-\infty) &= -T+(T^2-p)(T^2-2p)\\
 &=(2;T_2;T_3),\\ 
 \delta_p \left(\sum_{\substack{P \in C(\ol{\mathbb{Q}}) \\ x_{P}^{2}-p = 0}}P - 2\infty \right) &=(T^2-p)-T(T^2-2p) \\ 
 &=(p; T_2;p). 
 \end{align*} 

Hence, the above four elements lie in $\im(\delta_{p})$. 
By taking their $p$-adic valuations into account, we see that they are linearly independent.
By \cref{2 torsion21}, this completes the proof for $v=p$.

$v=\infty$:
\cref{image of 2-torsion} implies that 
\begin{align*}
 \delta_{\infty}((0,0)-\infty) &= -T+(T-\sqrt{p})(T+\sqrt{p})(T-\sqrt{2p})(T+\sqrt{2p})\\
 & =  (1;(-1,1);(-1,1)),\\
   \delta_{\infty} \left((-\sqrt{p},0)-\infty \right) &= -(T+\sqrt{p})+T(T-\sqrt{p})(T-\sqrt{2p})(T+\sqrt{2p})\\
   &= (-1;(-1,-1);(-1,1)).
    \end{align*} 

Hence, the above two elements lie in $\im(\delta_{\infty})$. 
We see that they are linearly independent. 
By \cref{2 torsion21}, this completes the proof for $v=\infty$.

$v=2$:
First, we show that the above four elements actually lie in $\im(\delta_{2})$.
\cref{image of 2-torsion} implies that
\begin{align*}
 \delta_2((0,0)-\infty) &= -T+(T^2-p)(T^2-2p)\\
 & =  (2;-T_2;-T_3),\\
 \delta_2 \left(\sum_{\substack{P \in C(\ol{\mathbb{Q}}) \\ x_{P}^{2}-p = 0}}P - 2\infty \right) &=(T^2-p)-T(T^2-2p)\\
   &   =  (3;T_2;-3).\\
 \end{align*} 
Hence, the above two elements lie in $\im(\delta_{2})$. 
\begin{itemize}
\item  Since $f(6)/2^2 \equiv 1 \pmod {8}$, there exists $P \in C(\mathbb{Q}_2)$ such that $x_P=6$. 
In particular, $(6;6-T_2;6-T_3)=\delta_2(P-\infty)$ lies in $\im(\delta_{2})$. 
\item  Suppose that
Since $f(5)/2^2 \equiv 1 \pmod{8}$ (resp.\ $f(13)/2^2 \equiv 1 \pmod{8}$), there exists $Q \in C(\mathbb{Q}_2)$ such that $x_Q=5$ (resp.\  $x_Q=13$)
if $p \equiv 13 \pmod{32}$ (resp.\ $p \equiv 29 \pmod{32}$).
In particular, $(5;5-T_2;5-T_3)=\delta_2(Q-\infty)$ (resp.\ $(5;13-T_2;13-T_3)=\delta_2(Q-\infty)$) lies in $\im(\delta_{2})$
if $p \equiv 13 \pmod{32}$ (resp.\ $p \equiv 29 \pmod{32}$). 
\end{itemize}
Since $v_2(2)=v_2(6)=1$, the first and the third elements and the last element are non-trivial in $L_2^{ \times}/L_2^{ \times 2}$. 
Since $3$ and $5$ is non-trivial in $\mathbb{Q}_2^{\times}/\mathbb{Q}_2^{\times 2}$, the second and the fourth element is  non-trivial in $L_2^{ \times}/L_2^{ \times 2}$. 

Finally, by taking the first and the second components into account, we see that they are linearly independent.  
By \cref{2 torsion21}, this completes the proof for $v=2$.
\end{proof}

To show \cref{Sel2}, we use \cref{LS2,Jv} and the following lemma.
\begin{lemma} \label{fu} 
Let $p \equiv 5 \pmod{8}$ be a prime number, 
 $K:=\mathbb{Q}(\sqrt{p})$ and $\epsilon \in \mathcal{O}_K$ be a fundamental unit.
 Then, $\epsilon$, $2 \epsilon \not\in \mathbb{Z}_p[\sqrt{p}]^{\times 2}$.
\end{lemma}

\begin{proof}
Let $\epsilon=(a+b\sqrt{p})/2$, $a$, $b \in \mathbb{Z}$.
Suppose that $\epsilon \in \mathbb{Z}_p[\sqrt{p}]^{\times 2}$. Then,
there exists $x$ such that $x^2 \equiv \epsilon \pmod{p}$.
Since 
\[\epsilon^2 \equiv \frac{a^2}{4} \equiv \frac{a^2-b^2p}{4}=N(\epsilon)=-1 \pmod{\sqrt{p}}, \]
$\epsilon \mod{\sqrt{p}}$ is a primitive fourth root of unity.
Thus, $x$ is a  primitive eighth root of unity in $\mathbb{Z}_p[\sqrt{p}]/(\sqrt{p}) \simeq \mathbb{F}_p$, which is a contradiction by $p \equiv 5 \pmod{8}$.
Therefore, $\epsilon \not\in \mathbb{Z}_p[\sqrt{p}]^{\times 2}$.
In exactly the same manner, we can prove that $2 \epsilon \not\in \mathbb{Z}_p[\sqrt{p}]^{\times 2}$.
\end{proof}

\begin{proof} [{Proof of \cref{Sel2} for Case (3)}]
Set elements of $L_2^{ \times}/L_2^{ \times 2} \times L_p^{\times}/L_p^{\times 2} \times L_{\infty}^{\times}/L_{\infty}^{\times 2}$ as follows: 
\begin{align*}  
d_1 &:=  (2;-T_2;-T_3)_2 \times\1_p \times \1_{\infty}, \\
d_2 &:= (3;T_2;-3)_2 \times\1_p \times \1_{\infty}, \\
d_3 &:= (6;6-T_2;6-T_3)_2 \times\1_p \times \1_{\infty}, \\
d_4 &:= \begin{cases}
(5;5-T_2;5-T_3)_2 \times\1_p \times \1_{\infty} & (p \equiv 13 \pmod{32}), \\
(5;13-T_2;13-T_3)_2 \times\1_p \times \1_{\infty} & (p \equiv 29 \pmod{32}), 
\end{cases}\\
d_5 &:=    \1_2 \times (2;T_2;T_3)_p \times \1_{\infty}, \\
d_6 &:= \1_2 \times   (p; T_2;p)_p \times \1_{\infty}, \\
d_7 &:=  \1_2 \times \1_p \times (1;(-1,1);(-1,1))_{\infty}, \\
d_8 &:=   \1_2 \times \1_p \times (-1;(-1,-1);(-1,1))_{\infty}, \\
h_1 &:=(1;1;-1)_2 \times \1_p \times (1;(1,1);(-1,-1))_{\infty}, \\
h_2 &:=(1;1;2)_2 \times (1;1;2)_p \times \1_{\infty}, \\ 
h_3 &:=(1;-1;1)_2 \times \1_p \times (1;(-1,-1);(1,1))_{\infty}, \\
h_4 &:=(1;\epsilon;\epsilon')_2 \times (1;\epsilon;\epsilon')_p \times (1;(1,-1);(1,-1))_{\infty}, \\
h_5 &:= (1;2;1)_2 \times (1;2;1)_p \times \1_{\infty}, \\
h_6 &:= (-1;\epsilon;1)_2 \times (1;\epsilon;1)_p \times (-1;(1,-1);(1,1))_{\infty}, \\
h_7 &:= (2;T_2;T_3)_2 \times (2;T_2;T_3)_p \times (1;(1,-1);(1,-1))_{\infty}, \\
h_8 &:= (3;T_2;1)_2 \times (p;T_2;1)_p \times (-1;(1,-1);(1,1))_{\infty}.
\end{align*}  
Note that
\begin{align*}
\res_S(t_1) &= (2;-T_2;-T_3)_2 \times(2;T_2;T_3)_p  \times (1;(-1,1);(-1,1))_{\infty} 
=d_5h_1h_3h_7=d_1d_5d_7,\\
\res_S(t_2) &=(3;T_2;-3)_2 \times (p; T_2;p)_p  \times (-1;(1,-1);(1,1))_{\infty}
=h_2h_8 =d_2d_6d_7d_8. 
\end{align*} 
Since
\begin{align*}
\Sel(\mathbb{Q},J) \simeq \{\alpha \in \Ker(N: L(S,2) \to \mathbb{Q}^{\times}/\mathbb{Q}^{\times 2}) \mid  \forall v \in S, \; \res_v(\alpha) \in \im(\delta_v)  \}
\supset <t_1,t_2>_{\mathbb{F}_2},
\end{align*} 
it is sufficient to prove that the nineteen elements $d_3, \ldots d_{8}, h_1, \ldots h_{8}$ are linearly independent over $\mathbb{F}_2$.  
Set
\begin{align*}
d_3 &:= (6;6-T_2;6-T_3)_2 \times\1_p \times \1_{\infty}, \\
d_4 &:= \begin{cases}
(5;5-T_2;5-T_3)_2 \times\1_p \times \1_{\infty} & (p \equiv 13 \pmod{32}), \\
(5;13-T_2;13-T_3)_2 \times\1_p \times \1_{\infty} & (p \equiv 29 \pmod{32}), 
\end{cases}\\
d_5 &:=    \1_2 \times (2;T_2;T_3)_p \times \1_{\infty}, \\
d_6 &:= \1_2 \times   (p; T_2;p)_p \times \1_{\infty}, \\
d_7 &:=  \1_2 \times \1_p \times (1;(-1,1);(-1,1))_{\infty}, \\
d_8 &:=   \1_2 \times \1_p \times (-1;(-1,-1);(-1,1))_{\infty}, \\
h_1 &:=(1;1;-1)_2 \times \1_p \times (1;(1,1);(-1,-1))_{\infty}, \\
h_2 &:=(1;1;2)_2 \times (1;1;2)_p \times \1_{\infty}, \\
h_3' &:=h_3d_7=(1;-1;1)_2 \times \1_p \times (1;(1,-1);(-1,1))_{\infty}, \\
h_4' &:=h_4d_7h_1h_3=(1;-\epsilon;-\epsilon')_2 \times (1;\epsilon;\epsilon')_p \times \1_{\infty}, \\
h_5 &:= (1;2;1)_2 \times (1;2;1)_p \times \1_{\infty}, \\
h_6' &:=h_6d_8h_1h_3h_4
=(-1;-1;-\epsilon')_2 \times (1;1;\epsilon')_p \times \1_{\infty}, \\
h_7' &:=h_7d_5d_7h_1h_3= (2;-T_2;-T_3)_2 \times \1_p \times \1_{\infty}, \\
h_8' &:=h_8d_3d_5d_6d_8h_1h_2h_3h_7= (1;T_2-6;2T_3(T_3-6))_2 \times \1_p \times \1_{\infty}. 
\end{align*}
It is sufficient to prove that the above nineteen elements are linearly independent.
\begin{enumerate}
\item By taking the first components at $v=\infty$ into account, we see that there is no relation containing $d_8$. 
\item By taking the first part of the second components at $v=\infty$ into account, we see that there is no relation containing $d_7$. 
\item By taking the second part of the second components at $v=\infty$ into account, we see that there is no relation containing $h_3'$.
\item By taking the first part of the third components at $v=\infty$ into account, we see that there is no relation containing $h_1$.

\item By taking the first components at $v=p$ into account, we see that there is no relation containing $d_5$ and $d_6$.
\item By taking the second components at $v=p$ into account, we see that there is no relation containing $h_4'$ and $h_5$ by \cref{fu}.

\item By taking the first components at $v=2$ into account, we see that there is no relation containing $d_3$, $d_4$, $h_6'$ and $h_7'$.
\item By taking the second components at $v=2$ into account, we see that there is no relation containing $h_8'$.
\item By taking the third components at $v=2$ into account, we see that there is no relation containing $h_2$. 
\end{enumerate}
\end{proof}


\subsection{Case (4): $p \equiv 5 \pmod{16}$ and $(i,j)=(2,3)$} 

Suppose that $(i,j)=(2,3)$, and $p \equiv 5 \pmod{16}$. 

\begin{lemma}  \label{2 torsion23}
The following two elements form an $\mathbb{F}_{2}$-basis of $J(\mathbb{Q})[2]$, $J(\mathbb{Q}_{2})[2]$ and $J(\mathbb{Q}_{p})[2]$: 
\begin{align*}
 (0,0) &- \infty, & \sum_{\substack{P \in C(\ol{\mathbb{Q}}) \\ x_{P}^{2}-p^3 = 0}}P &- 2\infty.  
 \end{align*}
On the other hand, the following three elements form an $\mathbb{F}_{2}$-basis of $J(\mathbb{Q}_{\infty})[2]$:
\begin{align*}
 (0,0) &- \infty, & (p\sqrt{p},0) &-\infty, & (-p\sqrt{p},0) &-\infty,  & (p\sqrt{2p}, 0)&-\infty.
 \end{align*}
In particular, we have the following table.
\begin{table}[ht]
    \begin{tabular}{|l||l|l|} \hline
  $v$   & $\dim J(\mathbb{Q}_{v})[2]$ & $\dim \im(\delta_{v})$ 
  \\ \hline \hline
     $2$   & $2$ & $4$      \\ \hline
     $p$   & $2$   & $2$       \\ \hline
      $\infty$   & $4$ & $2$   \\ \hline  
    \end{tabular}
\end{table}
\end{lemma}

\begin{proof}
The first and second statements follow from \cite[Lemma 5.2]{Stoll2014}.
Note that
\begin{itemize}
\item $p$ and $2p$ are not square in $\mathbb{Q}_{v}$ for $v=2$, $p$.
\item $2$ and $2p$ are square in $\mathbb{Q}_{\infty}$.
\end{itemize}
The third statement follows from the following formula (cf. \cite[p. 451, proof of Lemma 3]{FPS}).
\[\dim_{\mathbb{F}_2} \im(\delta_{v}) =\dim_{\mathbb{F}_2} J({\mathbb{Q}_{v}})[2]
\begin{cases}
+0 & (v \neq 2, \infty), \\
+2 & (v = 2), \\
-2 & (v = \infty). \\
\end{cases}\]
\end{proof}

\begin{proof}  [{Proof of \cref{Jv} for Case (4)}]
$v=p$: \cref{image of 2-torsion} implies that 
\begin{align*}
 \delta_p((0,0)-\infty) &= -T+(T^2-p^3)(T^2-2p^3)\\
 &=(2;T_2;T_3),\\ 
 \delta_p \left(\sum_{\substack{P \in C(\ol{\mathbb{Q}}) \\ x_{P}^{2}-p^3 = 0}}P - 2\infty \right) &=(T^2-p^3)-T(T^2-2p^3) \\ 
 &=(p; T_2;p).
 \end{align*} 

Hence, the above four elements lie in $\im(\delta_{p})$. 
By taking their $p$-adic valuations into account, we see that they are linearly independent.
By \cref{2 torsion23}, this completes the proof for $v=p$.

$v=\infty$: \cref{image of 2-torsion} implies that 
\begin{align*}
 \delta_{\infty}((0,0)-\infty) &= -T+(T-p\sqrt{p})(T+p\sqrt{p})(T-p\sqrt{2p})(T+p\sqrt{2p})\\
 & = (1;(-1,1);(-1,1)),\\
   \delta_{\infty}((-p\sqrt{p},0)-\infty) &= -(T+p\sqrt{p})+T(T-p\sqrt{p})(T-p\sqrt{2p})(T+p\sqrt{2p})\\
   &   = (-1;(-1,-1);(-1,1)).
 \end{align*} 

Hence, the above two elements lie in $\im(\delta_{\infty})$.
We see that they are linearly independent. 
By \cref{2 torsion23}, this completes the proof for $v=\infty$.

$v=2$: 
First, we show that the above four elements actually lie in $\im(\delta_{2})$.
\cref{image of 2-torsion} implies that
\begin{align*}
 \delta_2((0,0)-\infty) &= -T+(T^2-p^3)(T^2-2p^3)\\ 
 & =  (2;-T_2;-T_3),\\
 \delta_2 \left(\sum_{\substack{P \in C(\ol{\mathbb{Q}}) \\ x_{P}^{2}-p^3 = 0}}P - 2\infty \right) &=(T^2-p^3)-T(T^2-2p^3)\\
   &   =  (3;T_2;-3).\\
 \end{align*} 
Hence, the above two elements lie in $\im(\delta_{2})$. 
\begin{itemize}
\item  Since $f(6)/2^2 \equiv 1 \pmod {8}$, there exists $P \in C(\mathbb{Q}_2)$ such that $x_P=6$. 
In particular, $(6;6-T_2;6-T_3)=\delta_2(P-\infty)$ lies in $\im(\delta_{2})$. 
\item  Suppose that
Since $f(13)/2^2 \equiv 1 \pmod{8}$ (resp.\ $f(5)/2^2 \equiv 1 \pmod{8}$), there exists $Q \in C(\mathbb{Q}_2)$ such that $x_Q=13$ (resp.\  $x_Q=5$)
if $p \equiv 5 \pmod{32}$ (resp.\ $p \equiv 21 \pmod{32}$).
In particular, $(5;13-T_2;13-T_3)=\delta_2(Q-\infty)$ (resp.\ $(5;5-T_2;5-T_3)=\delta_2(Q-\infty)$) lies in $\im(\delta_{2})$
if $p \equiv 5 \pmod{32}$ (resp.\ $p \equiv 21 \pmod{32}$). 
\end{itemize}
Since $v_2(2)=v_2(6)=1$, the first and the third elements and the last element are non-trivial in $L_2^{ \times}/L_2^{ \times 2}$. 
Since $3$ and $5$ is non-trivial in $\mathbb{Q}_2^{\times}/\mathbb{Q}_2^{\times 2}$, the second and the fourth element is  non-trivial in $L_2^{ \times}/L_2^{ \times 2}$. 

Finally, by taking the first and the second components into account, we see that they are linearly independent.  
By \cref{2 torsion23}, this completes the proof for $v=2$.
\end{proof}


\begin{proof} 
Set elements of $L_2^{ \times}/L_2^{ \times 2} \times L_p^{\times}/L_p^{\times 2} \times L_{\infty}^{\times}/L_{\infty}^{\times 2}$ as follows: 
\begin{align*}  
d_1 &:=  (2;-T_2;-T_3)_2 \times\1_p \times \1_{\infty}, \\
d_2 &:= (3;T_2;-3)_2 \times\1_p \times \1_{\infty}, \\
d_3 &:= (6;6-T_2;6-T_3)_2 \times\1_p \times \1_{\infty}, \\
d_4 &:= \begin{cases}
(5;13-T_2;13-T_3)_2 \times\1_p \times \1_{\infty} & (p \equiv 5 \pmod{32}), \\
(5;5-T_2;5-T_3)_2 \times\1_p \times \1_{\infty} & (p \equiv 21 \pmod{32}), 
\end{cases}\\
d_5 &:=    \1_2 \times (2;T_2;T_3)_p \times \1_{\infty}, \\
d_6 &:= \1_2 \times   (p; T_2;p)_p \times \1_{\infty}, \\
d_7 &:=  \1_2 \times \1_p \times  (1;(-1,1);(-1,1))_{\infty}, \\ 
d_8 &:=   \1_2 \times \1_p \times  (-1;(-1,-1);(-1,1))_{\infty}, \\ 
h_1 &:=(1;1;-1)_2 \times \1_p \times (1;(1,1);(-1,-1)), \\ 
h_2 &:=(1;1;2)_2 \times (1;1;2)_p \times \1_{\infty}, \\
h_3 &:=(1;-1;1)_2 \times \1_p \times  (1;(-1,-1);(1,1)), \\ 
h_4 &:=(1;\epsilon;\epsilon')_2 \times (1;\epsilon;\epsilon')_p \times (1;(1,-1);(1,-1))_{\infty}, \\ 
h_5 &:= (1;2;1)_2 \times (1;2;1)_p \times \1_{\infty}, \\
h_6 &:= (-1;\epsilon;1)_2 \times (1;\epsilon;1)_p \times (-1;(1,-1);(1,1))_{\infty}, \\ 
h_7 &:= (2;T_2;T_3/p)_2 \times (2;T_2;T_3/p)_p \times (1;(1,-1);(1,-1))_{\infty}, \\ 
h_8 &:= (3;T_2;1)_2 \times (p;T_2;1)_p \times  (-1;(1,-1);(1,1))_{\infty}. 
\end{align*}  
Note that
\begin{align*}
\res_S(t_1) &= (2;-T_2;-T_3)_2 \times(2;T_2;T_3)_p  \times (1;(-1,1);(-1,1))_{\infty} 
=d_5h_1h_3h_7=d_1d_5d_7,\\
\res_S(t_2) &=(3;T_2;-3)_2 \times (p; T_2;p)_p  \times (-1;(1,-1);(1,1))_{\infty} 
=h_2h_8 =d_2d_6d_7d_8. 
\end{align*} 
Since
\begin{align*}
\Sel(\mathbb{Q},J) \simeq \{\alpha \in \Ker(N: L(S,2) \to \mathbb{Q}^{\times}/\mathbb{Q}^{\times 2}) \mid  \forall v \in S, \; \res_v(\alpha) \in \im(\delta_v)  \}
\supset <t_1,t_2>_{\mathbb{F}_2},
\end{align*} 
it is sufficient to prove that the nineteen elements $d_3, \ldots d_{8}, h_1, \ldots h_{8}$ are linearly independent over $\mathbb{F}_2$.  
Set
\begin{align*}
d_3 &:= (6;6-T_2;6-T_3)_2 \times\1_p \times \1_{\infty}, \\
d_4 &:= \begin{cases}
(5;13-T_2;13-T_3)_2 \times\1_p \times \1_{\infty} & (p \equiv 5 \pmod{32}), \\
(5;5-T_2;5-T_3)_2 \times\1_p \times \1_{\infty} & (p \equiv 21 \pmod{32}), 
\end{cases}\\
d_5 &:=    \1_2 \times (2;T_2;T_3)_p \times \1_{\infty}, \\
d_6 &:= \1_2 \times   (p; T_2;p)_p \times \1_{\infty}, \\
d_7 &:=  \1_2 \times \1_p \times  (1;(-1,1);(-1,1))_{\infty}, \\ 
d_8 &:=   \1_2 \times \1_p \times  (-1;(-1,-1);(-1,1))_{\infty}, \\ 
h_1 &:=(1;1;-1)_2 \times \1_p \times (1;(1,1);(-1,-1)), \\ 
h_2 &:=(1;1;2)_2 \times (1;1;2)_p \times \1_{\infty}, \\
h_3' &:=h_3d_7=(1;-1;1)_2 \times \1_p \times (1;(1,-1);(-1,1))_{\infty}, \\
h_4' &:=h_4d_7h_1h_3=(1;-\epsilon;-\epsilon')_2 \times (1;\epsilon;\epsilon')_p \times \1_{\infty}, \\
h_5 &:= (1;2;1)_2 \times (1;2;1)_p \times \1_{\infty}, \\
h_6' &:=h_6d_8h_1h_3h_4=(-1;-1;-\epsilon')_2 \times (1;1;\epsilon')_p \times \1_{\infty}, \\
h_7' &:=h_7d_5d_7h_1h_3= (2;-T_2;-T_3)_2 \times \1_p \times \1_{\infty}, \\
h_8' &:=h_8d_3d_5d_6d_8h_1h_2h_3h_7= (1;T_2-6;2T_3(T_3-6))_2 \times \1_p \times \1_{\infty}. 
\end{align*}
It is sufficient to prove that the above nineteen elements are linearly independent.
\begin{enumerate}

\item By taking the first components at $v=\infty$ into account, we see that there is no relation containing $d_8$. 
\item By taking the first part of the second components at $v=\infty$ into account, we see that there is no relation containing $d_7$. 
\item By taking the second part of the second components at $v=\infty$ into account, we see that there is no relation containing $h_3'$.
\item By taking the first part of the third components at $v=\infty$ into account, we see that there is no relation containing $h_1$.

\item By taking the first components at $v=p$ into account, we see that there is no relation containing $d_5$ and $d_6$.
\item By taking the second components at $v=p$ into account, we see that there is no relation containing $h_4'$ and $h_5$ by \cref{fu}.

\item By taking the first components at $v=2$ into account, we see that there is no relation containing $d_3$, $d_4$, $h_6'$ and $h_7'$.
\item By taking the second components at $v=2$ into account, we see that there is no relation containing $h_8'$.
\item By taking the third components at $v=2$ into account, we see that there is no relation containing $h_2$.
\end{enumerate}
\end{proof}

If we could reduce the calculation to $p=13$ (resp.\ $p=5$), the proofs of \cref{rank=0} 
would be short, but we couldn't. \\

\noindent {\bf Acknowledgements.}
The author thanks his advisor Kenichi Bannai for reading the draft and giving helpful comments. 
The author also thanks him for warm and constant encouragement, and
Yoshinosuke Hirakawa, Tatsuya Ohshita, Kazuki Yamada and Shuji Yamamoto for helpful comments and discussions.
The author thanks the the anonymous referee for her/his constructive comments in the original draft
to tell him the proof using the descent theorem for $p \equiv 3 \pmod{8}$ and $(i,j)=(0,2)$, $(2,2)$.

\end{document}